\documentclass[twoside,11pt]{amsart}

\usepackage{amsmath,latexsym,amssymb,mathptm, times}
\input amssym.def
\input amssym
\input xypic
\input xy
\xyoption{all}
\def\sqr#1#2{{\vcenter{\hrule height.#2pt
        \hbox{\vrule width.#2pt height#1pt \kern#1pt
                \vrule width.#2pt}
        \hrule height.#2pt}}}

\def\m{{\mathfrak m}}
\def\n{{\mathfrak n}}
\def\q{{\mathfrak q}}
\def\p{{\mathfrak p}}
\def\ms{\medskip}
\def\bs{\bigskip}
\def\s{\smallskip}
\def\ux{{\underline{x}}}
\newtheorem{Theorem}{Theorem}[section]

\newtheorem{Corollary}[Theorem]{Corollary}
\newtheorem{Proposition}[Theorem]{Proposition}

\newtheorem{Remark}[Theorem]{Remark}

\begin{document}

\baselineskip=16pt

\title[Integral extensions and the $a$-invariant]
{\bf Integral extensions and the $a$-invariant}

\author[A. Kustin, C. Polini, and B. Ulrich]
{ Andrew Kustin, Claudia Polini, \and Bernd Ulrich}

\thanks{AMS 2010 {\em Mathematics Subject Classification}.
Primary 13H10; Secondary 13A30, 13B22, 13C40, 13D45.}

\thanks{The first author was partially supported by the NSA.
The second author was partially supported by the NSF and the NSA.
The last author was partially supported by the NSF}

\thanks{The first author is the corresponding author.}


\address{Department of Mathematics, University of South Carolina,
Columbia, SC 29208} \email{kustin@math.sc.edu}

\address{Department of Mathematics, University of Notre Dame,
Notre Dame, Indiana 46556} \email{cpolini@nd.edu}

\address{Department of Mathematics, Purdue University,
West Lafayette, Indiana 47907} \email{ulrich@math.purdue.edu}

\vspace{-0.1in}

\begin{abstract}   In this
note we compare the $a$-invariant  of a homogeneous algebra $B$ to
the $a$-invariant   of a subalgebra $A$. In particular we show that if $A
\subset B$ is a finite homogeneous inclusion of standard graded
domains over an algebraically closed field with $A$ normal and $B$
of minimal multiplicity then $A$ has minimal multiplicity. In some
sense these results are algebraic generalizations of Hurwitz type theorems.
\end{abstract}

\maketitle

\vspace{-0.2in}

\section{Introduction}

The purpose of this note is to study the behavior of the
$a$-invariant under homogeneous inclusions of positively graded
algebras. Thus, let $k$ be a field and $A$ a positively graded
Noetherian $k$-algebra of dimension $d$ and with homogeneous maximal
ideal $\m$. The {\it $a$-invariant} $\, a(A)$ is the top degree of
the local cohomology module $H^d_{\m}(A)$ or, equivalently, the
negative of the initial degree of the graded canonical module
$\omega_A$. If $A$ is standard graded, then the $a$-invariant is
related to the Castelnuovo-Mumford regularity, via the inequality
$a(A) + d \leq {\rm reg} \, A$,
which is an equality in case $A$ is Cohen-Macaulay.

We will consider homogeneous inclusions $A \subset B$ of positively
graded algebras over a perfect field $k$. Our goal is to compare the
$a$-invariants of $A$ and of $B$. Most notably, we wish to bound the
$a$-invariant or the regularity of $A$ by means of the corresponding
invariants of $B$. We now describe what our results say in the
special case where the extension is integral and $B$ is a domain. If
${\rm char} \ k =0$ and $A$ is normal, then the trace map shows that
$A$ is a direct summand of $B$, hence $a(A) \leq a(B)$. This bound
still holds if the extension of quotient fields is separable, but
the estimates change sharply in the inseparable case. In the next
theorem we summarize what we prove, though not in the most general
form.


\begin{Theorem}\label{blah} Let $k$ be a perfect field and
$A \subset B$ a homogeneous integral extension of positively graded
Noetherian $k$-domains. Assume in addition that $A$ is regular in
codimension one.
\begin{enumerate}
\item If the field extension ${\rm
Quot}(A) \subset {\rm Quot}(B)$ is separable, then $a(A) \leq a(B)$.
\item If ${\rm char} \, k = p > 0 \, $, $d=2$, and $A$ is standard
graded, then $\lfloor \frac{a(A)}{p^e} \rfloor \leq a(B)$, where
$p^e$ is the inseparable degree of the field extension; equality
holds if $B$ is regular in codimension one and the field extension
is purely inseparable.
\item If ${\rm char} \, k = p > 0\, $ and $A$, $B$  are standard
graded, then $\lfloor \frac{a(A)+ d -2}{p^e} \rfloor \leq {\rm reg}
\, B -2$, where $p^e$ is the inseparable degree of the field
extension.
\end{enumerate}
\end{Theorem}

\ms

Notice that the bound in item (3), the inseparable case, is
considerable weaker than the estimate in (1), the separable case,
but it is nevertheless sharp according to (2). No such bounds can be
expected if $A$ is not regular in codimension one, regardless of
characteristic. For instance, let $R=k[x,y]$ be a polynomial ring
over any field and let $B$ be the $n$th Veronese subring of $R$,
i.e., the $k$-subalgebra of $R$ generated by all monomials of degree
$n$. Consider the subring $A=k[x^n,x^{n-1}y,y^n]$ of $B$.  After
regrading, $A \subset B$ is a homogeneous integral extension of
standard graded $k$-algebras. One has $a(B) = -1$, whereas
$a(A)=n-3$, which can be arbitrarily large. The first equality holds
because $B$ is a Veronese subring, whereas the second one follows
from the fact that $A$ is a hypersurface ring and has multiplicity
$n$.

Our proofs are, in some sense, algebraic generalizations of
arguments that can be found in \cite[Chap. IV, Sec. 2]{Ha}, for
instance. To prove part (1) of Theorem~\ref{blah} we produce
an embedding of graded canonical modules $\omega_A \hookrightarrow
\omega_B$. Part (3) is deduced from item (2) using an induction
argument, and (2) in turn follows once we have shown that
$\omega_{A^{(p^e)}} \hookrightarrow \omega_{B^{p^e}}$. Here $B^{p^e}$ denotes
the $p^e$th Frobenius power of $B$ (in other words, 
$B^{p^e}$ is the subring $\{b^{p^e}\mid b\in B\}$ of $B$) and 
$A^{(p^e)}$   the $p^e$th Veronese subring $A$.  Quite generally, if $C$ is a
positively graded Noetherian algebra over a field, then the $n$th
Veronese subring $C^{(n)}$ is defined as the graded subalgebra
$\oplus_i C_{ni}$ of $C$. The formation of Veronese subrings
commutes with taking local cohomology, hence $a(C^{(p^e)})=
p^e\lfloor \frac{a(C)}{p^e} \rfloor$. On the other hand, when $C$ is
a domain and has characteristic $p>0$, then $a(C^{p^e})=p^e a(C)$.
Thus, the inequality of Theorem \ref{blah}(2) follows once we have
established the embedding
$\omega_{A^{(p^e)}} \hookrightarrow \omega_{B^{p^e}}$.


The origin of this note was a problem that came up in our earlier
work \cite{KPU2014}, namely the question of whether minimal
multiplicity descends under integral extensions. Recall that a
Noetherian standard graded domain $A$ over an algebraically closed
field is said to have {\it minimal multiplicity} if its multiplicity
has the smallest possible value, namely the embedding codimension of
$A$ plus one. Minimal multiplicity is equivalent to the bound ${\rm
reg} \, A \leq 1$, which in turn holds if and only if $A$ is
Cohen-Macaulay with $a(A) \leq 1 - d$, where $d = {\rm dim} \ A$.
For $d=2$ the inequality $a(A) \leq 1 -d$ means that $[\omega
_A]_0=0$. On the other hand, the vector space dimension of
$[\omega_A]_0$
is the genus of $\mathcal C = {\rm Proj}(A)$, at least when $\mathcal
C$ is nonsingular. Thus, the homogeneous coordinate ring of a
nonsingular curve in projective space has minimal multiplicity if
and only if the genus of the curve is zero. In this setting, the
descent of minimal multiplicity follows from the classical Hurwitz
Formula that describes the change in genus under finite separable
morphisms of curves. Our results apply to any homogeneous integral
extension $A \subset B$ of standard graded Cohen-Macaulay domains
over an algebraically closed field as long as $A$ is normal, and
they show, in particular, that if $B$ has minimal multiplicity, then
so does $A$. Again, this implication no longer holds without the
normality assumption on $A$. It would be interesting to generalize
the descent of minimal multiplicity from the graded to the local
case.



\bs

\bs

\section{The Separable Case}\ms 

In this section we compare the $a$-invariants of $A\subset B$ when  suitable separability assumptions are in place. 



 \bs


Let $k$ be an infinite perfect field and $A$ a positively graded, reduced, Noetherian, equidimensional
 $k$-algebra of dimension $d$. Then there is a natural homogeneous map
\[
c_A: \wedge^d \Omega_{A/k} \longrightarrow \omega_A \, ,
\]
called the {\it canonical class}, that is an isomorphism off the
singular locus of $A$ (see \cite{El, KW, L}); here $\Omega$ denotes the module of differentials.

\ms

We will often use the following  fact from dimension theory.
\begin{Remark}\label{dim}{\rm If $k$ is a field and $A \subset B$ is an
inclusion of finitely generated $k$-algebras, then ${\rm dim} \,A
\leq {\rm dim} \, B$. To prove this one reduces to the case when $A$ and $B$ are domains; then one uses the fact that the Krull dimension equals the transcendence degree over $k$.} 
\end{Remark}

Recall that if $A$ is a Noetherian ring of dimension $d<\infty$, then $A^{\rm unm}$ stands for the ring $A/\cap Q_i$, where the $Q_i$ are the primary components of $(0)$ of dimension $d$.

\ms
\begin{Theorem} \label{sepcase}
Let $k$ be a perfect field and $A \subset B$ a homogeneous inclusion
of positively graded Noetherian $k$-algebras of dimension $d$.
Assume every minimal prime $\p_{i}$ of $A$ of  dimension $d$ is the
contraction of a minimal prime $\q_{i}$ of $B$ so that the extension
of quotient fields ${\rm Quot}(A/\p_{i}) \subset {\rm
Quot}(B/\q_{i})$ is separable. Further suppose that ${\rm dim} \,
B/(\cap \q_i+JB) \leq d - 2$ for some $A$-ideal $J$ with ${\rm
Sing}(A^{\rm unm}) \subset V(J A^{\rm unm})$. Then there is a
homogeneous embedding of canonical modules
\[
\omega_A \hookrightarrow \omega_B \, .
\]
In particular,
\[
a(A) \leq a(B)
\] and, if $A$ and $B$ are homogeneous and $A$ is Cohen-Macaulay,
\[
{\rm reg} \ A \leq a(B)+ d \leq {\rm reg} \ B \, .
\]
\end{Theorem}

\begin{proof} Notice that $A/\p_i  \hookrightarrow   B/\q_i$, thus  ${\rm dim} \, B/\q_i =d$ by Remark \ref{dim}. Since  ${\rm dim} \, B/(\cap \q_i+JB) <  {\rm dim} \, B/\q_i$ it follows that $J$ cannot be contained in any $\q_i$, hence is not
in any $p_i$. Therefore $A^{\rm unm}$ is reduced and  $A^{\rm unm}=A/\cap \p_i$. Notice that $\omega_A =\omega_{A^{\rm unm}}$ and $ \omega_{B/\cap\q_i} \hookrightarrow\omega_B $.
Thus we may replace $A \subset B$  by  $A^{\rm unm}= A/\cap \p_i \subset B/\cap \q_i$ to assume that $A$ and $B$ are reduced and  equidimensional with minimal primes $\{\p_i\}$ and $\{\q_i\}$, respectively.
Furthermore  one has ${\rm Quot}(A)=\times K_i \subset {\rm Quot} (B) =\times L_i$, where $K_i={\rm Quot}(A/\p_i) \subset L_i ={\rm Quot} (B/\q_i)$ are separable algebraic field extensions.

Let $J'$ be the Jacobian ideal of the $k$-algebra $B$. By our
assumption on $J$ and since $B$ is reduced, we have that ${\rm dim}
\, B/JJ' \leq d-1$. Thus ${\rm dim }\, A/ ((JJ') \cap A) \leq {\rm
dim}\, B/ JJ' \leq d-1$ where the first inequality obtains
by Remark \ref{dim}. It follows that
the homogeneous ideal $JJ' \cap A$ cannot be in any minimal prime of
$B$ because any such prime contracts to a prime of $A$ of dimension
$d$. Hence by prime avoidance there exists a homogeneous element $f
\in (JJ')\cap A$ that is a non zerodivisor on $B$. Notice that $A_f$
and $B_f$ are regular rings. Hence $(c_A)_f$ is an isomorphism,
which gives a commutative diagram
$$
{ \diagram
\  \wedge^d\Omega_{A/k}  \ar[r]^{c_A} \ \dto^{\rm nat}  & \ \,
 \omega_A \dto|<\hole|<<\bhook^{\rm nat}  \\
\  (\wedge^d \Omega_{A/k})_f \ & \quad \ \,  (\omega_A)_f  \quad  \ar[l]^{\sim}_{(c_A)_f^{-1} }
\enddiagram }
$$
where the right most map is an embedding since $f$ is $A$-regular.
We also have a corresponding diagram for $B$.

Now there is a commutative diagram of natural homogeneous linear maps,
$$
{ \diagram
\ \wedge^d\Omega_{A/k}   \ar[r]^{\hspace{ -.6in}\epsilon}\ \  \dto^{c_A} & B \otimes \wedge^d \Omega_{A/k} \cong \wedge^d(B \otimes_A \Omega_{A/k})  \quad \ar[r]^{\hspace{.8in} \mu}  \  &  \ \,
\wedge^d \Omega_{B/k}  \dto^{ c_B}  \\
 \omega_A \ar[drr]^{\varphi} \dto|<\hole|<<\bhook^{\rm nat}& & \omega_B  \dto|<\hole|<<\bhook^{\rm nat} \\
\  (\omega_A)_f \cong(\wedge^d \Omega_{A/k})_f     \ar[rr]^{\psi} & &  (\wedge^d \Omega_{B/k})_f\cong  (\omega_B)_f
\enddiagram }
$$
where $\psi=(\mu \epsilon)_f$ and $\varphi$ is the restriction of $\psi$.

We show that $\psi$ is injective. Indeed, $A_f$ is a regular ring
and therefore the $A_f$-module $(\wedge^d \Omega_{A/k})_f \cong
\wedge^d \Omega_{A_f/k}$ is projective, hence torsionfree. Thus it
suffices to show that $K \otimes_{A_f} \psi$ is injective for
$K={\rm Quot}(A_f)={\rm Quot}(A)$. Since ${\rm Quot}(A)=\times K_i
\subset L={\rm Quot} (B) =\times L_i$ with $K_i \subset L_i $
separable algebraic field extensions, we have natural isomorphisms
$L_i \otimes_{K_i} \Omega_{K_i/k} \cong \Omega_{L_{i}/k} $ and
therefore $L\otimes_K \Omega_{K/k} \cong \Omega_{L/k}$.  It follows
that $L \otimes \wedge^d \Omega_{K/k} \cong \wedge^d (L \otimes_K
\Omega_{K/k})\cong \wedge^d \Omega_{L/k}$. Clearly $\wedge^d
\Omega_{K/k}$ embeds into $L \otimes_K \wedge^d \Omega_{K/k}$
because $L$ is a faithfully flat $K$-module, for instance. In
summary, we obtain the commutative diagram
$$
{ \diagram \ \wedge^d\Omega_{K/k}   \ar[r]^{K \otimes \psi}\ \
\dto|<\hole|<<\bhook & \wedge^d \Omega_{K \otimes_A
B/k}  \dto  \\
\  L \otimes_K \wedge^d \Omega_{K/k}    \ar[r]^{\sim} & \wedge^d
\Omega_{L/k},
\enddiagram }
$$
which shows that $K \otimes \psi$ is injective.
Hence
indeed, $\psi$ is injective.

Therefore the composition $\varphi$ is a homogeneous embedding of
graded $A$-modules. It remains to show that ${\rm im} \varphi \subset
\omega_B$. However, since ${\rm Sing}(A) \subset V(J)$, some power
$I$ of $J$ annihilates the cokernel of the canonical class $c_A$.
Thus $I\cdot {\rm im} \varphi \subset  \omega_B$ by the commutativity
of the above diagram. On the other hand,
since ${\rm dim} B/I \leq d-2,$ we have ${\rm depth}_{IB} \omega_B
\geq 2 $.  This implies that ${\rm im} \varphi \subset \omega_B$.
 \end{proof}

\ms

\begin{Remark}{\rm The assumption
on $J$ obtains
if $A$ is regular in codimension one and $B$ is integral over $A$.
If $A$ is equidimensional, then the ideal $J$ can be taken to be the Jacobian ideal of $A$, 
that is, the $d$-th Fitting ideal of $\Omega_{A/k}$.}
\end{Remark}

\ms

\begin{Remark}{\rm If $A \subset B$ is an integral extension,
one can also obtain Theorem~\ref{sepcase} using the trace map
$\overline{B} \longrightarrow \overline{A}$ and dualizing this map
into $\omega_A$.}
\end{Remark}

\ms

The assumption that ${\rm dim}\ A ={\rm dim}\ B$ in
Theorem~\ref{sepcase} can be removed if $B$ is standard graded.

 \ms
\begin{Remark}
{\rm Let $k$ be a perfect field and $A \subset B$ a homogeneous
inclusion of positively graded Noetherian $k$-algebras. Assume that
$B$ is standard graded. Further suppose every minimal prime $\p_i$
of $A$ of maximal dimension is the contraction of a minimal prime
$\q_i$ of $B$ of maximal dimension so that the extension of quotient
fields ${\rm Quot}(A/\p_i) \subset {\rm Quot}(B/\q_i)$ is separable.
In addition assume that ${\rm dim} B/(\cap \q_i +JB) \leq {\rm dim}
\ B -2$ for some $A$-ideal $J$ with ${\rm Sing}(A^{\rm unm}) \subset
V(J A^{\rm unm})$. Then there is a homogeneous embedding of
canonical modules
\[ \omega_A({\rm dim} \ A) \hookrightarrow \omega_B ({\rm dim} \
B).
\]
In particular,
\[
a(A) + {\rm dim}\, A \leq a(B) +{\rm dim}\,  B
\] and, if $A$ is standard graded and Cohen-Macaulay, then
\[
{\rm reg} \ A \leq a(B) + {\rm dim} \, B \leq {\rm reg} B.
\]
\begin{proof}We may suppose that $k$ is infinite and as in the proof of Theorem~\ref{sepcase} we reduce to
the case where $A$ and $B$ are reduced and equidimensional with
minimal primes $\{\p_i\}$ and $\{\q_i\}$, respectively.
Write $\delta= {\rm dim} \, B - {\rm dim} \, A$ and let $x_1, \ldots
x_{\delta}$ be general linear  forms in $B$. One has $\delta = {\rm
dim} \, B/\q_i - {\rm dim} \, A/\p_i ={\rm trdeg}_{A/\p_i} \, B/\q_i
$ for each of the finitely many minimal primes $\q_i$. Thus the
images
of $x_1, \ldots, x_{\delta}$ in $B/\q_i$ forms a separating
transcendence basis of ${\rm Quot}(B/\q_i)$ over ${\rm
Quot}(A/\p_i)$. Since $B$ is reduced, this also shows that the
linear forms  $x_1, \ldots, x_{\delta}$ are algebraically
independent over $A$. Therefore $\omega_A=\omega_{A[x_1, \ldots,
x_{\delta}]}(\delta)$. Now, replacing $A$ by the graded polynomial
ring $A[x_1, \ldots, x_{\delta}]$ we may assume that $d= {\rm dim}
\, A={\rm dim} \, B$. Notice that the separability assumption is
preserved by the definition of a separating transcendence basis. The
assertion now follows from Theorem~\ref{sepcase}.
\end{proof} }
\end{Remark}

%


\section{The Two dimensional Case}
\ms

In this section we allow for inseparable extensions, but we require the dimension of the rings to be $2$.

\begin{Theorem}\label{dim2} Let $k$ be a perfect field of positive characteristic $p$ and $A \subset B$ a
homogeneous inclusion of positively graded Noetherian $k$-algebras
of dimension two. Assume $A$ is a domain containing a nonzero linear
form. Further suppose that ${\rm dim} \, B/(\q+J B) \leq 0$ for some
$A$-ideal $J$ with ${\rm Sing}(A) \subset V(J )$ and some minimal
prime $\q$ of $B$ such that $\q \cap A =0$.

The field extension ${\rm Quot}(A) \subset {\rm Quot}(B/\q)$ is
algebraic; write $p^e$ for its inseparable degree.
There is a homogeneous embedding of canonical modules
\[
\omega_{A^{(p^e)}} \hookrightarrow \omega_{B^{p^e}}.
\]
In particular,
\[\left
\lfloor \frac{a(A)}{p^e} \right\rfloor  \leq a(B)
\] and, if $A$ is standard graded and Cohen-Macaulay,
\[\left
\lfloor \frac{{\rm reg } \, A  -2}{p^e}\right \rfloor \leq a(B).
\]
\end{Theorem}

\begin{proof} As $A \hookrightarrow B/\q$,
Remark \ref{dim} gives that ${\rm dim} \, B/\q \geq {\rm dim} \, A$,
and therefore ${\rm dim} \, B /\q =2={\rm dim} \, A$. In particular,
the field extension ${\rm Quot}(A) \subset {\rm Quot}(B/\q)$ is
algebraic. Since furthermore $\omega_{B/\q} \hookrightarrow
\omega_{B}$ we may replace $B$ by $B/\q$ to assume that $B$ is a
domain. Write $K={\rm Quot}(A)$, $L={\rm Quot}(B)$, and let $K_{\rm
sep}$ be the separable closure of $K$ in $L$. One has $L^{p^e}
\subset K_{\rm sep}$, which gives $K_{\rm sep}=KL^{p^e}$ because the
extension $KL^{p^e} \subset L$ is purely inseparable. Since $k$ is
perfect and $L$ has transcendence degree $2$ over $k$, there exist
elements $u, v$ in $L$  so that the extension $k(u,v) \subset L$ is
separable  algebraic. Therefore $L=k(u,v)L^{p^e}=L^{p^e}(u,v)$,
where the last equality uses the perfection of $k$ again. We
conclude that the extension $L^{p^e} \subset L$ has degree at most
$p^{2e}$. Let $x$ be a nonzero linear form in $A$. Notice that the
set $\{x^i \mid 0 \leq i \leq p^e-1\}$ forms a  basis of $K$ over
${\rm Quot}(A^{(p^e)})$ and is linearly independent over the ring
$B^{(p^e)}$, which in turn contains $A^{(p^e)}[B^{p^e}]$. Thus the
purely inseparable field extensions ${\rm Quot}(A^{(p^e)}) \subset
K$ and ${\rm Quot}(A^{(p^e)})L^{p^e} \subset KL^{p^e}$ have degrees
$p^e$ and $\geq p^e$, respectively. We summarize our findings in the
following diagram of field extensions and their respective degrees,
\begin{equation}\label{dia1}
{ \diagram
K    \rto|<\hole|<<\ahook_{\rm sep.} & K_{\rm sep}=KL^{p^e}   \rto|<\hole|<<\ahook^{\ \ \qquad  p^e} &  L  \\
{\rm Quot}(A^{(p^e)}) \rto|<\hole|<<\ahook  \uto|<\hole|<<\ahook_{p^e} & {\rm Quot}(A^{(p^e)})L^{p^e} \uto|<\hole|<<\ahook^{\rm p. insep.}_{\geq p^e} &\\
&  L^{p^e}    \uto|<\hole|<<\bhook   \uurto|<\hole|<<\bhook_{\leq p^{2e}} &
\enddiagram }
\end{equation}

\s \noindent Calculating inseparable degrees in the square on the
left shows that the extension ${\rm Quot}(A^{(p^e)}) \subset {\rm
Quot}(A^{(p^e)})L^{p^e}$ is separable, and computing field degrees
along the triangle on the right hand side we see that $L^{p^e}  =
{\rm Quot}(A^{(p^e)})L^{p^e} $.

Since the extension $A^{(p^e)}[B^{p^e}]\subset B$ is integral we
have that $J^{p^e}$ generates an ideal of dimension at most zero in
$A^{(p^e)}[B^{p^e}]$. Now applying Theorem~\ref{sepcase} to the
inclusion $A^{(p^e)} \subset A^{(p^e)}[B^{p^e}]$ yields an embedding
of canonical modules $\omega_{A^{(p^e)}} \hookrightarrow
\omega_{A^{(p^e)}[B^{p^e}]}$. On the other hand, since $B^{p^e}
\subset A^{(p^e)}[B^{p^e}]$ is a finite and birational extension we
obtain an inclusion $\omega_{A^{(p^e)}[B^{p^e}]}\hookrightarrow
\omega_{B^{p^e}}$.

Finally, we prove the inequality $\lfloor \frac{a(A)}{p^e} \rfloor
\leq a(B).$
One has $a(A^{(p^e)})=p^e \lfloor \frac{a(A)}{p^e}\rfloor,  $ since
the local cohomology functors commute with the formation of Veronese
submodules. On the other hand, $a(B^{p^e} )= p^e a(B)$.
\end{proof}

\ms
The next proposition shows that the estimate of Theorem~\ref{dim2} is sharp
and that the stronger inequality of Theorem~\ref{sepcase} is not valid
without the separability assumption.
\s

\begin{Proposition}
Let $k$ be a perfect field of positive characteristic $p$ and $A
\subset B$ a homogeneous integral extension of positively  graded
Noetherian $k$-domains of dimension two. Assume $A$ is normal and
contains a nonzero linear form. Further suppose that the field
extension ${\rm Quot}(A) \subset {\rm Quot}(B)$  is purely
inseparable of degree $p^e$. Then $B^{p^e} \subset A^{(p^e)}$ is a
finite birational extension. In particular, $B$ is normal if and
only if $B^{p^e} =A^{(p^e)}$, in which case
\[\left\lfloor \frac{a(A)}{p^e} \right\rfloor  = a(B).
\]
\end{Proposition}
\begin{proof} Write $K={\rm Quot}(A)$ and $L={\rm Quot}(B)$.
One has $B^{p^e} \subset L^{p^e} \subset K$ and therefore $B^{p^e}
\subset K \cap B \subset A$ where the last inclusion obtains because
$B$ is integral over $A$ and $A$ is normal. It follows that $B^{p^e}
\subset A^{(p^e)}$. This extension is integral since $A^{(p^e)}
\subset B$, and it is birational because the proof of
Theorem~\ref{dim2} shows that ${\rm Quot} (A^{(p^e)})L^{p^e}=
L^{p^e},$ hence ${\rm Quot} (A^{(p^e)})\subset L^{p^e}.$
\end{proof}

\bs

\section{The General Case}
\ms

In this section we continue our treatment of possibly inseparable extensions in positive characteristic. We allow the dimensions of the rings to be arbitrary, but we do not recover the full strength of the result in dimension $2$.
The $t$-th Fitting ideal of a finitely presented $C$-module $M$ is denoted by ${\rm Fitt}_{t}^CM$.

\begin{Theorem}\label{gen} Let $k$ be a perfect field of positive characteristic $p$ and let $A \subset B$ a
homogeneous integral extension of standard graded Noetherian
$k$-algebras. Assume that $A$ is a domain, regular in codimension
one, and write $p^e$ for the inseparable degree of the field
extension ${\rm Quot}(A) \subset {\rm Quot} (B/\q),$ where $\q$ is
some minimal prime of $B$ of maximal dimension. Then
\[\left
\lfloor \frac{a(A)+{\rm dim}\, A -2}{p^e} \right\rfloor \leq {\rm reg} \, B -2.
\]
\end{Theorem}
\begin{proof}
We may assume that $k$ is infinite. Write $d={\rm dim}\,  A={\rm
dim} \, B$. We prove the theorem by induction on $d$. If $d \leq 1$
then $A$ is regular, thus $\lfloor \frac{a(A)+{\rm dim}\, A -2}{p^e}
\rfloor =\lfloor \frac{-2}{p^e} \rfloor$. On the other hand, $ {\rm
reg} \, B -2 \geq -1$ unless ${\rm reg} \, B =0$, in which case
$B=A$ and hence $p^e=1$. For $d=2$ the assertion follows from
Theorem~\ref{dim2} because $a(A)+ {\rm dim} \, A \leq {\rm reg} \,
A$ and $a(B) \leq {\rm reg} \, B -2$.

Now assume that $d \geq 3$. Let ${\ell}_1, \ldots, {\ell}_n$ be
linear forms generating the homogeneous maximal ideal of $A$ and let
$z_1, \ldots , z_n$ be indeterminates over $k$. In the ring $A
\otimes _kk(z_1, \ldots , z_n)$ we consider the generic linear form
$x= \sum z_i {\ell}_i$. Without changing the notation introduced in
the theorem, we now replace $k$ by the purely transcendental
extension field $k(z_1, \ldots, z_n)$. Write $\m$ and $\n$ for the
homogeneous maximal ideals of $A$ and $B$, respectively. Notice that
$\sqrt{\m B}=\n$. Consider the induced homogeneous module finite map
of standard graded $k$-algebras
$$
\overline{A}=(A/xA)/H^0_{\m}(A/xA) \longrightarrow \overline{B}=(B/xB)/H^0_{\m}(B/xB).
$$
Both algebras have dimension $d-1$. Since $x$ is a generic linear
form, the proof of \cite[Theorem]{H}
shows that $\overline{A}$ is a domain and is regular in codimension
one. It follows that the above map is an embedding $\overline{A}
\hookrightarrow \overline{B}$. As $d-1> 0$ and $x$ is $A$-regular it
follows that $a(\overline{A})=a(A/xA) \geq a(A)+1$. On the other
hand, the natural map $B  \twoheadrightarrow \longrightarrow
\overline{B}$ factors through the ring $B'=B/H^0_{\m}(B)$, and the
linear form $x$ is regular on the latter. As
$\overline{B}=(B'/xB')/H^0_{\m}(B'/xB')$ it follows that ${\rm reg}
\, \overline{B} \leq {\rm reg} \, B'/xB'= {\rm reg} \, B' \leq {\rm
reg} \, B$. Hence the statement of the theorem  follows from the
induction hypothesis once we have shown  that there exists a minimal
prime $\overline{\q}$ of $\overline{B}$ of maximal dimension so that
the field extension ${\rm Quot}(\overline{A}) \subset {\rm
Quot}(\overline{B}/\overline{\q})$ has inseparable degree at most
$p^e$.

To construct $\overline{\q}$ notice that the $B$-ideal $(\q, xB)$
has dimension $d-1$ and hence so does its saturation $(\q,xB):
\m^{\infty}$. Let $\q'$ be a prime ideal of dimension $d-1$
containing the latter ideal and define $\overline{\q}=
\q'\overline{B}$. Clearly $\overline{\q}$ is a prime ideal of
$\overline{B}$ of maximal dimension and it contains the image of
$\q$. Write $K={\rm Quot}(A)$, $L={\rm Quot}(B/\q)$, let $K_{\rm
sep}$ be the separable closure of $K$ in $L$, and define $C=K_{\rm
sep} \cap (B/\q )$. As $L=K[B/\q]$ one has $K_{\rm sep}={\rm Quot}(
C)$. Thus $B/\q$ is a $C$-module of rank $p^e$. Set $\q''=(\q'/\q)
\cap C$, which is a $(d-1)$-dimensional prime ideal of $C$
containing $x$. Consider the $C$-ideal $J={\rm
Fitt}_0^C(\Omega_{C/A}) \cdot {\rm Fitt}_{p^e}^C(B/\q)$.  This ideal
is not zero because the field extension ${\rm Quot}( A)=K \subset
{\rm Quot} (C)=K_{\rm sep}$ is separable algebraic and because
$B/\q$ has rank $p^e$ as $C$-module. Hence there are at most
finitely many $(d-1)$-dimensional prime ideals containing $J$, and
$x$ cannot be contained in any of them, since $x$ is generic for
$\m$ and $\m$ generates an ideal of dimension $0 < d-1$ in $C$. It
follows that $J$ cannot be in  $q''$,  a $(d-1)$-dimensional prime
containing $x$. Therefore $(\Omega_{C/A})_{\q''}=0$ and
$(B/\q)_{\q''}$ is a free $C_{\q''}$-module of rank $p^e$. Tensoring
with the residue field $k(\q'')$ of $\q''$ we conclude that
$\Omega_{C/A} \otimes_C k(\q'')=0$ and $B/\q \otimes_C k(\q'')$ is a
$k(\q'')$-vector space of dimension $p^e$. Finally, write
$\overline{C}= C/ \q''$, which yields the inclusions of domains
$\overline{A} \subset \overline{C} \subset
\overline{B}/\overline{\q}$. Notice that ${\rm
Quot}(\overline{C})=k(\q'')$. As $\Omega_{{\overline C}/{\overline
A}} \otimes_{\overline C}k(\q'')$ and $\overline{B}/\overline{\q}
\otimes_{\overline C}k(\q'')$ are epimorphic images of $\Omega_{C/A}
\otimes_C k(\q'')$ and $B/\q \otimes_C k(\q'')$, respectively, we
conclude that  $\Omega_{{\overline C}/{\overline A}}
\otimes_{\overline C}{\rm Quot}(\overline{C})=0 \,$ and $\,
\overline{B}/\overline{\q}\otimes_{\overline C}{\rm
Quot}(\overline{C})$  is a ${\rm Quot}(\overline{C})$-vector space
of dimension at most  $p^e$.  Therefore the field extension ${\rm
Quot}(\overline{A}) \subset {\rm Quot}(\overline{C})$ is separable
and the extension ${\rm Quot}(\overline{C}) \subset {\rm Quot}
(\overline{B}/\overline{\q})$ has degree at most $p^e$. It follows
that the inseparable degree of ${\rm Quot}(\overline{A}) \subset{\rm
Quot}(\overline{B}/\overline{\q})$ is at most $p^e$.
\end{proof}

 \bs

 \section{Applications}

\ms

Now we come to the original goal of our work, which is showing that minimal multiplicity descends under integral extensions. 

Let $k$ be a field and $A$ a quasi-standard graded $k$-algebra, by
which we mean a positively graded Noetherian $k$-algebra integral
over a subalgebra generated by linear forms. Write $\m$ for the
homogeneous maximal ideal of $A$. Assume that either $A$ is
Cohen-Macaulay or else $k$ is algebraically closed and $A$ is a
domain. In this case $e(A) \geq {\rm edim} \, A - {\rm dim} \, A +1$,
\cite{A} and \cite[pg.~112]{EG}. If equality holds one says that $A$ has {\it minimal
multiplicity}.  It is known that $A$ has minimal multiplicity if and
only if $A$ is Cohen-Macaulay and $\m ^2 \subset J$ for some (every)
ideal $J$ generated by a linear system of parameters. In case $A$ is
standard graded the following conditions are equivalent as well
\cite[Introduction]{EG}:

\begin{itemize}
\item $A$ has minimal multiplicity
\item ${\rm reg} \ A \leq 1$
\item $A$ is Cohen-Macaulay and $a(A) \leq 1- {\rm dim} \ A$
\item $A$ is Cohen-Macaulay and $\m ^2 \subset J$ for some (every)
ideal $J$ generated by a linear system of parameters.
\end{itemize}

Integral extensions of such algebras are somewhat restricted. Thus
let $k$ be an algebraically closed field and $A$ a standard graded
$k$-domain of minimal multiplicity. The integral closure $\overline
A$ of $A$ is a quasi-standard graded $k$-algebra. Since the
extension $A \subset \overline A$ is birational, one has
$e({\overline A}) \leq e(A) $. Thus one concludes that $${\rm edim}
\, \overline A \leq e( \overline A) + {\rm dim} \, \overline A - 1
\leq e({ A}) + {\rm dim} \, { A}  -1 = {\rm edim} \, { A}
 \, .$$
On the other hand, as $A$ is standard graded, a homogeneous minimal
generating set of the homogeneous maximal ideal of $A$ extends to a
homogeneous minimal generating set of the homogeneous maximal ideal
of $\overline A$.  Thus the inequality ${\rm edim} \, \overline A
\leq {\rm edim} \, {A}$ implies that $A= \overline A$. This recovers
the well known fact that any standard graded domain over an
algebraically closed field is normal, provided it has minimal
multiplicity (much more is true, see for instance
\cite[19.9]{Harris}).

This discussion shows that in the next Corollary it is natural to assume that $A$ is normal. This condition is approximated by the assumptions of Theorems~\ref{sepcase},
\ref{dim2}, or \ref{gen}. 

\ms

\begin{Corollary}
In addition to the assumptions of either Theorems~\ref{sepcase},
\ref{dim2}, or \ref{gen}, suppose that $A$ and $B$ are standard
graded and $A$ is Cohen-Macaulay. Further assume that $B$ is
Cohen-Macaulay or that $k$ is algebraically closed and $B$ is a
domain. If $B$ has minimal multiplicity so does $A$.
\end{Corollary}
\begin{proof} The theorems show that ${\rm reg} \ A \leq 1$ if
${\rm reg} \, B \leq 1$.
\end{proof}

\ms
In particular, we obtain the following statement that we use in \cite{KPU2014}.

\ms

\begin{Corollary}\label{6.25} Let $k$ be a perfect field and $A\subset B$ homogeneous integral extensions of standard graded $k$-domains. Further assume that $A$ is normal and Cohen-Macaulay. If $B$ has minimal multiplicity, then so does $A$. \end{Corollary}

\ms
We finish by recording a curious fact. 

\begin{Proposition}\label{6.3}
Let $k$ be a field and $A  \varsubsetneq B$ a proper homogeneous
integral extension of positively graded Cohen-Macaulay $k$-algebras.
If $A$ is a standard graded ring of minimal multiplicity then $B/A$
is a maximal Cohen-Macaulay $A$-module.
\end{Proposition}
\begin{proof} We may assume that $k$ is infinite. Let $\m$ be the homogeneous maximal ideal of
$A$. Let $\ux=x_1, \ldots ,x_d$ be linear forms that are a system of
parameters of $A$ and let $J$ be the $A$-ideal they generate. Write
$C=B/A$. To show that $\ux$ form a regular sequence on $C$ we need
to prove that the first Koszul homology of $\ux$ with coefficients
in $C$ vanishes or, equivalently, ${\rm Tor}_1^A(C,A/J)=0$. The
vanishing of Tor is equivalent to the equality $JB\cap A =J$, as can
be seen from the exact sequence
\[
0={\rm Tor}_1^A(B, A/J) \longrightarrow {\rm Tor}_1^A(C,A/J)
\longrightarrow A/J \longrightarrow B/JB \longrightarrow C/JC
\longrightarrow 0.
\]
Thus let $\alpha$ be a homogeneous element of $JB \cap A$. If
$\alpha$ has degree one then $\alpha \in J$ because $B_0=k=A_0$. If
on the other hand, $\alpha$ has degree at least two, then $\alpha
\in \m^2 \subset J$, where the last inclusion holds because $A$ has
minimal multiplicity.
\end{proof}
 \ms


\bs

\end{document}